\documentclass[12pt]{article}
\usepackage{indentfirst, latexsym, bm, amssymb}

\begin{document}

\newtheorem{theorem}{Theorem}[section]
\newtheorem{corollary}[theorem]{Corollary}
\newtheorem{definition}[theorem]{Definition}
\newtheorem{conjecture}[theorem]{Conjecture}
\newtheorem{question}[theorem]{Question}
\newtheorem{lemma}[theorem]{Lemma}
\newtheorem{proposition}[theorem]{Proposition}
\newtheorem{example}[theorem]{Example}
\newenvironment{proof}{\noindent {\bf
Proof.}}{\rule{3mm}{3mm}\par\medskip}
\newcommand{\remark}{\medskip\par\noindent {\bf Remark.~~}}
\newcommand{\pp}{{\it p.}}
\newcommand{\de}{\em}

\newcommand{\JEC}{{\it Europ. J. Combinatorics},  }
\newcommand{\JCTB}{{\it J. Combin. Theory Ser. B.}, }
\newcommand{\JCT}{{\it J. Combin. Theory}, }
\newcommand{\JGT}{{\it J. Graph Theory}, }
\newcommand{\ComHung}{{\it Combinatorica}, }
\newcommand{\DM}{{\it Discrete Math.}, }
\newcommand{\ARS}{{\it Ars Combin.}, }
\newcommand{\SIAMDM}{{\it SIAM J. Discrete Math.}, }
\newcommand{\SIAMADM}{{\it SIAM J. Algebraic Discrete Methods}, }
\newcommand{\SIAMC}{{\it SIAM J. Comput.}, }
\newcommand{\ConAMS}{{\it Contemp. Math. AMS}, }
\newcommand{\TransAMS}{{\it Trans. Amer. Math. Soc.}, }
\newcommand{\AnDM}{{\it Ann. Discrete Math.}, }
\newcommand{\NBS}{{\it J. Res. Nat. Bur. Standards} {\rm B}, }
\newcommand{\ConNum}{{\it Congr. Numer.}, }
\newcommand{\CJM}{{\it Canad. J. Math.}, }
\newcommand{\JLMS}{{\it J. London Math. Soc.}, }
\newcommand{\PLMS}{{\it Proc. London Math. Soc.}, }
\newcommand{\PAMS}{{\it Proc. Amer. Math. Soc.}, }
\newcommand{\JCMCC}{{\it J. Combin. Math. Combin. Comput.}, }
\newcommand{\GC}{{\it Graphs Combin.}, }

\title{Complete multipartite graphs are determined by their distance spectra \thanks{
 Supported by National Natural Science Foundation of China
(No. 11271256). \newline \indent $^{\dagger}$Correspondent author:
Xiao-Dong Zhang (Email: xiaodong@sjtu.edu.cn)}}
\author{  Ya-Lei Jin  and Xiao-Dong Zhang$^{\dagger}$   \\
{\small Department of Mathematics, and Ministry of Education }\\
{\small Key Laboratory of Scientific and Engineering Computing, }\\
{\small Shanghai Jiao Tong University} \\
{\small  800 Dongchuan road, Shanghai, 200240, P.R. China}\\
}

\maketitle
 \begin{abstract}
  It is well known that the complete multipartite graphs can not be determined by their
  adjacency spectra. But in this paper, we prove that they can be
  determined by their distance
  spectra, which confirms the conjecture  proposed by  Lin, Hong, Wang and
  Shu.
   \end{abstract}


\vskip 0.5cm

\section{Introduction}

 In this paper, we only consider simple and undirected connected graphs.
 Let $G=(V(G), E(G))$ be a graph of order $n$ with vertex
set $V(G)$ and edge set $E(G)$. Let $d_G(v)$ and $N_G(v)$ denote the
degree  and neighbors of a vertex $v $, respectively.
 $D(G)=(a_{uv})_{n \times n}$ denotes the distance matrix of $G$ with
$a_{uv}=d_G(u, v)$, where $d_G(u, v)$ is the distance between
vertices $u$ and $v$.  The largest eigenvalue of $D(G)$, denoted by
 $\lambda(G)$, is
called the distance spectral radius of $G$.
 The research for distance matrix can
be dated back to the paper \cite{Edelberg}, which presents an
interesting result that is the determinant of the distance matrix of
trees with order $n$ is always $(-1)^{n-1}(n-1)2^{n-2}$, independent
of the structure of the tree. Recently, the distance matrix of a
graph has received increasing attention. For example, Liu \cite{Liu}
characterized the graphs with minimal spectral radius of the
distance matrix in three classes of simple connected graphs  with
fixed vertex connectivity, matching number and chromatic number,
respectively. Zhang \cite{Zhang} determined the unique graph with
minimum distance spectral radius among all connected graphs  with a
given diameter. Bose, Nath and Paul \cite{Bose} characterized the
graph with minimal distance spectral radius among all graphs with the
fixed number of pendent vertices.

   Denote by $Sp(D(G))$ the set of all eigenvalues of $D(G)$ including
the multiplicity. Two graphs $G$ and $G'$ are called D-cospectral if
$Sp(D(G))=Sp(D(G'))$.  A graph $G$ is determined by its D-spectral
if $Sp(D(G))=Sp(D(G'))$  implies  $G\cong G'$. There are three
excellent surveys (\cite{Dam1}, \cite{Dam2}  and \cite{Dam3})  that
which graphs can be determined by their spectra.

 Lin et al.\cite{Lin2} characterized all
  the connected graphs with smallest eigenvalue is $-2$,
   and  proposed  the following conjecture:
\begin{conjecture}\label{C1}\cite{Lin2}
Let $G=K_{n_1,n_2,...,n_k}$ be a complete k-partite graph. Then $G$ is determined by its D-spectrum.
\end{conjecture}
Moreover, they proved that the conjecture is true for $k=2$. On the
other hand, it is well known that complete multipartite graphs can
not be determined by their spectra. For example, $K_{1,4}$ is not
determined by its spectra, since $K_{1,4}$ and the union of cycle of
order 4 and an isolated vertex have the same spectrum but not isomorphic.  In this paper, we will give a
positive answer for this conjecture.
\section{\large\bf{Main Results}}

Before the proof of the conjecture, we need the following theorems and lemmas.
\begin{theorem}\label{Lmain}\cite{Lin2}
Let $G$ be a connected graph and $D$ be the distance matrix of $G$.
Then the smallest eigenvalue of $D(G)$ is equal to $-2$ with
multiplicity $n-k$ if and only if $G$
 is a complete k-partite graph for $2\le k\le n-1$.
\end{theorem}

\begin{theorem}\label{LK}\cite{Lin2}
Let $G=K_{n_1,...,n_k}$ be a complete k-partite graph. Then the Characteristic polynomial of $D(G)$ is
$$P_D(\lambda)=(\lambda+2)^{n-k}\left[\prod_{i=1}^k(\lambda-n_i+2)-
\sum_{i=1}^kn_i\prod_{j=1,j\neq i}^k(\lambda-n_j+2)\right].$$
\end{theorem}

On symmetric function, we follow notations in the Chapter $7$ of the
book \emph{Enumerative Combinatorics}\cite{Stanley}.  For any given
nonnegative vectors $a=(a_1,...,a_k),x=(x_1,...,x_k)$, $c_i$ is a
real number and $x^a=x_1^{a_1}x_2^{a_2}...x_k^{a_k}$.
\begin{definition}
Let $x=(x_1,x_2,...,x_k)\in R^k$, $f(x)$ is called the symmetric function of degree $d$ if
$$f(x_1,...,x_k)=\sum_{(a_1,...,a_k)\in \mathcal{V}} c_{a_1,...,a_k} x_1^{a_1}...x_k^{a_k}$$
where $\mathcal{V}$ is the set of $k$ dimensional nonnegative integer vectors, $ c_{a_1,...,a_k} \in R$ and $f(x)$ satisfies the following conditions:\\
(a)$f(x_1,...,x_i,...,x_j,...,x_k)=f(x_1,...,x_j,...,x_i,...,x_k),1\le i<j\le k$;\\
(b) $d=\max_{a\in \mathcal{V}}\{a_1+a_2+...+a_k\}$.
\end{definition}
Let
$$\psi_t(x)=\sum_{j=1}^kx_j^t.$$  It is well known that
\begin{theorem}\label{DT}(\cite{Stanley})
For any symmetric function $f(x)$, there exists a unique function
$F(x)=\sum_{a\in \mathcal{V}} c_a x^a$ ,where $\mathcal{V}$ is the
subset of $k$ dimensional nonnegative integer vectors, such that
$f(x)=F(\psi_1(x),\psi_2(x),...,\psi_k(x))$.
\end{theorem}

\begin{lemma}\label{XS}
Let\begin{eqnarray*} g(\lambda)&:=&
\prod_{i=1}^k(\lambda-n_i+2)-\sum_{i=1}^kn_i\prod_{j=1,j\neq
i}^k(\lambda-n_j+2)\\
& :=& \sum_{i=0}^k\xi_i\lambda^{k-i}
\end{eqnarray*}
 with $\sum_{i=1}^kn_i=n.$ Then there exists a unique function
 $F^{(i)}(y)=\sum_{a\in \mathcal{V}_i} c_a y^a$ with $y=(y_1,\cdots,
 y_i)$
    such that
$$\xi_i=F^{(i)}(\psi_1(n_1,...,n_k),\psi_2(n_1,...,n_k),...,\psi_i(n_1,...,n_k)),
\ \ i=1,\cdots, k,$$
 where $\mathcal{V}_i$ is the subset of $i$ dimensional nonnegative integer
 vectors. Moreover, there are some nonzero constants $c_i$
such that $F^{(i)}(y_1,y_2,...,y_i)-c_iy_i$  is a function on
variables $y_1, \cdots, y_{i-1}$ for $i=1,\cdots, k$, i.e., $
F^{(i)}(y_1,y_2,...,y_i)=c_iy_i+G^{(i)}(y_1, \cdots, y_{i-1})$.
\end{lemma}
\begin{proof}
Clearly, $\xi_0=1,$ and $\xi_1=-2\psi_1(n_1,...,n_k)+2k$.
 For $1<i\le k$,  it follows from the definition that
\begin{eqnarray*}
\xi_i&=&(-1)^i\sum_{1\le j_1<j_2<...<j_i\le k}(n_{j_1}-2)(n_{j_2}-2)...(n_{j_i}-2)\\
&-&\sum_{l=1}^kn_l(-1)^{i-1}\sum_{1\le j_1<j_2<...<j_{i-1}\le
k,l\notin \{j_1,...,
j_{i-1}\}}(n_{j_1}-2)(n_{j_2}-2)...(n_{j_{i-1}}-2)\\
&=&(-1)^i\sum_{1\le j_1<j_2<...<j_i\le k}(n_{j_1}-2)(n_{j_2}-2)...(n_{j_i}-2)\\
&&+(-1)^{i}\sum_{l=1}^k\sum_{1\le j_1<j_2<...<j_{i-1}\le k,l\notin
\{j_1,
...,j_{i-1}\}}(n_{j_1}-2)(n_{j_2}-2)...(n_{j_{i-1}}-2)(n_l-2)\\
&&+2(-1)^{i-1}\sum_{l=1}^k\sum_{1\le j_1<j_2<...<j_{i-1}\le
k,l\notin
\{j_1,...,j_{i-1}\}}(n_{j_1}-2)(n_{j_2}-2)...(n_{j_{i-1}}-2).\\
&=&(-1)^i(i+1)\sum_{1\le j_1<j_2<...<j_i\le k}(n_{j_1}-2)(n_{j_2}-2)...(n_{j_i}-2)\\
&&+(-1)^{i-1}2(k+1-i)\sum_{1\le j_1<j_2<...<j_{i-1}\le
k}(n_{j_1}-2)(n_{j_2}-2)...(n_{j_{i-1}}-2).
\end{eqnarray*}
Hence  $\xi_i$ is a symmetric function  of degree $i$ on the variables
$(n_1,...,n_k)$ for $i=1, \cdots, k.$ Then, by Theorem \ref{DT},
there exists a unique function $F^{(i)}(y)=\sum_{a\in
\mathcal{V}_i}c_ay^a$ on $y=(y_1, \cdots, y_i)$ such that
$$\xi_i=F^{(i)}(\psi_1(n_1,...,n_k),\psi_2(n_1,...,n_k),...,\psi_i(n_1,...,n_k)).$$
Moreover, the coefficient in single term of $\xi_i$ with maximum
degree $i$ is $(-1)^i(i+1)$. Hence the coefficient of $F(y)=\sum_{a\in
\mathcal{V}_i}c_ay^a$ on variable $y_i$ is a constant. Hence  the
assertion holds.
\end{proof}

\begin{lemma}\label{MA}
Let\begin{eqnarray*} g(\lambda)&:=&
\prod_{i=1}^k(\lambda-n_i+2)-\sum_{i=1}^kn_i\prod_{j=1,j\neq
i}^k(\lambda-n_j+2)\\
& :=& \sum_{i=0}^k\xi_i\lambda^{k-i}
\end{eqnarray*}
 with $\sum_{i=1}^kn_i=n.$
 Then there exists a unique function
$H^{(i)}(y)=\sum_{a\in \mathcal{V}} c_a y^a$  such that
$\psi_i(n_1,...,n_k)=H^{(i)}(\xi_1(n_1, \cdots, n_k),\xi_2(n_1,
\cdots, n_k), \cdots, \xi_i(n_1, \cdots, n_k)$, where
$\mathcal{V}_i$ is the subset of $i$ dimensional nonnegative integer
vectors,$i=1, \cdots, k.$
\end{lemma}
\begin{proof}
We prove the assertion by the induction on $i$. For $i=1$, clearly,
there exists a function $H^{(1)}(y)=H^{(1)}(y_1)=-\frac{1}{2}y_1+k$
such that $\psi_1(n_1, \cdots,n_k)= -\frac{1}{2}\xi_1+k$ since
$\xi_1=-2\psi_1+2k$. Assume that the assertion holds for $i=1,
\cdots, t$. In other words, there exist $H^{(1)}(y), \cdots,
H^{(t)}(y)$ such that $\xi_j(n_1, \cdots, n_k)=H^{(j)}(\xi_1,
\cdots, \xi_j)$ for $j=1, \cdots, t$. For $i=t+1$, by
Lemma~\ref{XS}, there exists a unique function
$F^{(t+1)}(y)=\sum_{a\in \mathcal{V}_{t+1}}c_ay^a$ with $y=(y_1,
\cdots, y_{t+1})$ such that
\begin{eqnarray*}
\xi_{t+1}(n_1, \cdots, n_k)&=&F^{(t+1)}(\psi_1(n_1, \cdots, n_k),
\cdots, \psi_{t+1}(n_1, \cdots, n_k))\\
&=& c_{t+1}\psi_{t+1}(n_1, \cdots, n_k)+G^{(t+1)}(\psi_1(n_1,
\cdots, n_k), \cdots, \psi_t(n_1, \cdots, n_k)). \end{eqnarray*}
Hence
$$\psi_{t+1}(n_1, \cdots, n_k)=-\frac{1}{c_{t+1}}\xi_{t+1}(n_1, \cdots, n_k)-
G^{(t+1)}(\psi_1(n_1, \cdots, n_k), \cdots, \psi_t(n_1, \cdots,
n_k)).$$ Hence there exists a unique function
$H^{(t+1)}(y)=\sum_{a\in \mathcal{V}} c_a y^a$  such that
$\psi_{t+1}(n_1, \cdots, n_k)=H^{(i)}(\xi_1(n_1, \cdots,
n_k), \xi_2(n_1, \cdots, n_k), \cdots, \xi_{t+1}(n_1, \cdots, n_k)).$
So the assertion holds.
\end{proof}

\begin{lemma}\label{maintheorem}
For any given  vector $z=(z_1,z_2,...,z_k)$,  if the equations system
$$\psi_t(x)=z_t, \ \ i.e., \sum_{i=1}^kx_i^t= z_t, t=1, \cdots,  k$$
with unknown variables $x_1, \cdots, x_k$ has two nonnegative solutions with
$x=(x_1, \cdots, x_k)=(a_1, \cdots, a_k)$ and $x=(b_1,\cdots, b_k)$,
then there is a permutation $\sigma$ such that $a=(a_1, \cdots,
a_k)=(b_{\sigma(1)}, \cdots, b_{\sigma(n)})$.
\end{lemma}
\begin{proof}
 We prove the assertion by the induction on $k$.
For  $k=2$,   it is easy to the system $x_1+x_2=z_1,
x_1^2+x_2^2=z_2$ with unknown
variables $x_1, x_2$ has two solutions with
 $(x_1,x_2)=(\frac{z_1+\sqrt{4z_2-z_1^2}}{2},\frac{z_1-\sqrt{4z_2-z_1^2}}{2})$  and
  $(x_1, x_2)=(\frac{z_1+\sqrt{4z_2-z_1^2}}{2},\frac{z_1+\sqrt{4z_2-z_1^2}}{2})$.
  Hence the assertion holds.
 Assume that the assertion holds for the number less than $k$.
  Moreover, assume that the equations
system
$$\psi_t(x)=z_t, i=1, \cdots,  k$$
with unknown variables $x_1, \cdots, x_k$ has two nonnegative solutions with
$x=(x_1,\cdots,x_k)=(a_1, \cdots, a_k)$ and $x=(b_1,...,b_k)$. Then
we claim that there exist $1\le i, j\le k$ such that $a_i=b_j$. In
fact, otherwise, without loss of generality, assume that
$a_1=\cdots=a_p>a_{p+1}\ge\cdots\ge a_k$ and $a_1>b_1$. Then
 $\psi_t(a)=\psi_t(b)$ for $t=1, \cdots, k.$ Since $\psi_t(x)$ is
 symmetric function, by Theorem \ref{DT}, there exists a  unique
 function $F(x)$ such that $\psi_t(x)=F(\psi_1, \cdots, \psi_k)$ for all $t\ge 1$.
 Hence $\psi_t(a)=\psi_t(b)$ holds for all $t\ge 1$.
  Then
  $$p+\sum_{i=p+1}^k\left(\frac{a_i}{a_1}\right)^t=
  \sum_{i=1}^k\left(\frac{b_i}{a_1}\right)^t,$$
  for all $t\ge 1$.
  Therefore
$$p+\sum_{i=p+1}^k\lim_{t\rightarrow \infty}\left(\frac{a_i}{a_1}\right)^t=
\lim_{t\rightarrow
\infty}\sum_{i=1}^n\left(\frac{b_i}{a_1}\right)^t,$$ which implies
that $p=0$. It is impossible. So the claim holds. Therefore there
exist $1\le i, j\le k$ such that $a_i=b_j$. Moreover, it is easy to
see that the equations system $\sum_{l=1}^{k-1}x_l^t=z_t-a_i^t$ for
$t=1, \cdots, k-1$ has two solutions $(x_1, \cdots, x_{k-1})= (a_1,
\cdots, a_{i-1}, a_{i+1}, \cdots, a_k)$ and $(x_1, \cdots, x_{k-1})=
(b_1, \cdots, b_{j-1}, b_{j+1}, \cdots, b_k)$. By the induction
hypothesis, there exists a permutation $\sigma_1$ such that $(b_1,
\cdots, b_{j-1}, b_{j+1}, \cdots, b_k)=(a_{\sigma_1(1)}, \cdots
a_{\sigma_1(i-1)}, a_{\sigma_1(i+1)}, \cdots, a_{\sigma_1(k)})$.
Hence the assertion holds.
\end{proof}

Now we are ready to present our main theorem

\begin{theorem}\label{C}\cite{Lin2}
Let $G=K_{n_1, n_2, \cdots, n_k}$ be a complete $k-$partite graph
with $sum_{i=1}^kn_i=n$. Then $G$ is determined by its D-spectrum.
\end{theorem}
\begin{proof}Let $G^{\prime}$ be any simple graph with
$sp(D(G^{\prime}))=Sp(D(G))$. Since $G$ is complete $k-$partite
graph, by Theorem~\ref{Lmain}, $G^{\prime}$ is complete $k-$partite
graphs, since the smallest eigenvalue of $D(G)$ and $D(G^{\prime})$
are equal to $-2$ with multiplicity $n-k$. Hence assume that
$G^{\prime}=K_{p_1, \cdots, p_k}$ with $\sum_{i=1}^kp_i=n$. Let the
function \begin{eqnarray*} f_i(x_1, \cdots, x_k)&:=&
(-1)^i(i+1)\sum_{1\le j_1<j_2<...<j_i\le k}(x_{j_1}-2)(x_{j_2}-2)...(x_{j_i}-2)\\
&+&(-1)^{i-1}2(k+1-i)\sum_{1\le j_1<j_2<...<j_{i-1}\le
k}(x_{j_1}-2)(x_{j_2}-2)...(x_{j_{i-1}}-2),
\end{eqnarray*}
for $i=1, \cdots, k$.  By Lemma~\ref{XS}, there exists a unique
function $H^{(i)}(y)$  with $y=(y_1, \cdots, y_i)$ such that
$$\psi_i(x)=\sum_{j=1}^kx_j^i=H^{(i)}(f_1(x_1, \cdots, x_k), \cdots,
f_i(x_1, \cdots, x_k))$$ for $i=1, \cdots, k$.
 Since  the coefficients of the characteristic polynomial of distance
 matrices of $D(G)$ and $D(G^{\prime})$ are equal, we have
 $g_G(\lambda)=g_{G^{\prime}}(\lambda)$. Hence
 $f_i(n_1, \cdots, n_k)=f_i(p_1, \cdots, p_k):=\theta_i$.
 Therefore
 $$\psi_i(n_1, \cdots, n_k)=H^{(i)}(f_1(n_1,\cdots, n_k), \cdots,
 f_i(n_1, \cdots,n_k))=H^{(i)}(\theta_1, \cdots, \theta_k)$$
 and
  $$\psi_i(p_1, \cdots, p_k)=H^{(i)}(f_1(p_1,\cdots, p_k), \cdots,
 f_i(p_1, \cdots,p_k))=H^{(i)}(\theta_1, \cdots, \theta_k)$$
 for $i=1, \cdots, k$. Hence the equations system
 $\psi_i(x_1, \cdots, x_k)=H^{(i)}(\theta_1, \cdots, \theta_k)$,
 $i=1, \cdots, k$ with unknown variables $x_1, \cdots, x_k$ has two
 solutions $(x_1, \cdots, x_k)=(n_1, \cdots, n_k)$ and
 $(x_1, \cdots, x_k)=(p_1, \cdots, p_k)$. By
 Lemma~\ref{maintheorem}, there exists a permutation $\sigma$ such
 that $(n_1, \cdots, n_k)=(p_{\sigma(1)}, \cdots, p_{\sigma(k)})$.
 Hence $G^{\prime}=K_{p_1, \cdots, p_k}=K_{n_1, \cdots, n_k}$. The
 assertion holds.
\end{proof}

\end {document}